\documentclass[12pt]{amsart}
\usepackage{amsthm, amssymb, amsfonts}
\usepackage{bbm} 
\usepackage[pdfborder={0 0 0}]{hyperref}
\usepackage{color}
\usepackage{enumerate}
\usepackage[capitalise]{cleveref}
\usepackage{autonum}
\usepackage[foot]{amsaddr}

\makeatletter
\newcommand{\restore@Environment}[1]{%
  \AtBeginDocument{%
    \csletcs{#1*}{#1}%
    \csletcs{end#1*}{end#1}%
  }%
}
\forcsvlist\restore@Environment{alignat,equation,gather,multline,flalign,align}
\makeatother

\renewcommand{\epsilon}{\varepsilon}

\newcommand{\n}{\mathbb{N}}
\newcommand{\z}{\mathbb{Z}}
\newcommand{\q}{\mathbb{Q}}
\renewcommand{\r}{\mathbb{R}}

\newcommand{\f}{\mathbb{F}}
\renewcommand{\a}{\mathbb{A}}
\newcommand{\dd}{\; \mathrm{d}}

\renewcommand{\vec}[1]{\underline{#1}}
\renewcommand{\phi}{\varphi}
\newcommand{\pdv}[2]{\frac{\partial #1}{\partial #2}}
\newcommand{\one}{\mathbbm{1}}

\DeclareMathOperator{\Mod}{mod}

\DeclareMathOperator{\Sing}{Sing}

\theoremstyle{plain} 
\newtheorem{thm}{Theorem}[section]
\newtheorem{lem}[thm]{Lemma} 
\newtheorem{cor}[thm]{Corollary} 
\newtheorem{prop}[thm]{Proposition}

\theoremstyle{definition}

\theoremstyle{remark}
\newtheorem{remark}[thm]{Remark}

\title[Quantitative Results on Diophantine Equations]{Quantitative Results on Diophantine Equations in~Many~Variables}
\author{Jan-Willem M. van Ittersum}
\address{\scriptsize Mathematisch Instituut, Universiteit Utrecht, Postbus 80.010, 3508 TA Utrecht, Nederland}
\email{j.w.m.vanittersum@uu.nl}
\thanks{I would like to thank my supervisor Damaris Schindler for helpful discussions and suggesting research questions and ideas.}
\date{\today}

\begin{document}

\begin{abstract}
\begin{sloppypar}
We consider a system of integer polynomials 
of the same degree with non-singular local zeros and in many variables. Generalising the work of Birch \cite{Bir62} we find a quantitative asymptotic formula (in terms of the maximum of the absolute value of the coefficients of these polynomials) for the number of integer zeros of this system within a growing box. Using a quantitative version of the Nullstellensatz, we obtain a quantitative strong approximation result, i.e. an upper bound on the smallest non-trivial integer zero provided the system of polynomials is non-singular.
\end{sloppypar}
\end{abstract}

\maketitle


\section{Introduction}
Consider a system $\vec{f}$ of polynomials $f_1, \ldots, f_r\in \z[x_1, \ldots, x_n]=\z[\vec{x}]$ of degree $d\geq 2$. It was shown by Birch \cite{Bir62}, that if these polynomials are homogeneous, they satisfy the {smooth Hasse principle} providing \begin{equation}\label{eq:nlarge}
n-\dim V^*>r(r+1)(d-1)2^{d-1}
,\end{equation} where $V^*$ is the so-called Birch singular locus
of the the projective variety $V$ corresponding to $\vec{f}$. Let $V^{sm}$ be the smooth locus of $V$ (which consists of the points where the the Jacobian matrix of $\vec{f}$ has rank $r$).
Then, the system $\vec{f}$ is said to satisfy the \emph{smooth Hasse principle} if
$$\prod_{\nu} V^{sm}(\q_\nu) \neq \emptyset \quad \text{implies that} \quad V(\q)\neq \emptyset.$$
Here, the product is over all places $\nu$ of $\q$ and $\q_\infty=\r$. 

In this paper we are interested in the distribution and the size of the rational points on $V$ (or integer points on $V$ when the system is not assumed to be homogeneous). More specifically, let $\mathcal{V}_{\z}$ be an integral model of $V$. Let $\a^\infty$ be the adele ring of $\q$ outside the place $\infty$ and let $\mathcal{V}_{\a^\infty}$ be the base change of $\mathcal{V}_{\z}$ to $\a^\infty$. We say that $V$ satisfies \textit{strong approximation outside $\infty$} if the image of the diagonal map
$\mathcal{V}_{\z}\to \mathcal{V}_{\a^\infty}$
is dense.  Note that the notion of strong approximation outside $\infty$ implies the smooth Hasse principle.
For $V$ as in Birch's theorem strong approximation outside $\infty$ holds. \cref{thm:cormain} is a quantitative version of this statement, which is a first step in understanding the distribution of the integer zeros of arbitrary systems of integer polynomials. This result follows directly from our main theorem stated in \cref{sec:mr}. 
 In order to obtain this result we generalise the work of Birch \cite{Bir62} to find quantitative asymptotics (in terms of the maximum of the absolute value of the coefficients of these polynomials) for the number of integer zeros of this system within a growing box. Using a quantitative version of the Nullstellensatz, we obtain an upper bound on the smallest non-trivial common zero of $\vec{f}$. 

\subsection{Related work}\label{sec:rw} There are many improvements on Birch's result if we restrict to a single form. For example, Heath-Brown showed that a cubic form has a non-trivial integer zero provided only that $n\geq 14$ \cite{Hea07}. Assuming that the variety $V$ is non-singular, a form of degree $2$, $3$, or $4$ satisfies the smooth Hasse principle provided that $n\geq 3$, $n\geq 9$ or $n\geq 40$ respectively \cite{Hea96,Hoo88,Han35}. Browning and Prendiville slightly relaxed condition (\ref{eq:nlarge}), by showing that for a form of degree $d\geq 3$ the smooth Hasse principle holds provided that $n-\dim V^* \geq \left(d-\tfrac{1}{2}\sqrt{d}\right)2^d$ \cite{BP}. 

Recent results by Myerson improve on Birch's result for systems of forms when $V$ is a complete intersection (which is implied by (\ref{eq:nlarge}) in Birch's theorem). He shows that under this condition one can replace condition (\ref{eq:nlarge}) by $n\geq 9r$ respectively $n\geq 25r$ for systems of degree $2$ respectively $3$ \cite{Mye15,Mye17}. 

Unconditional improvements include the observation that $\dim V^*$ can replaced by a smaller quantity $\Delta$, to be defined in (\ref{eq:defK}),
as shown independently by Dietmann and Schindler \cite{Die15,Sch14}. Another improvement is the observation by Schmidt that the assumption that the system of polynomials is homogeneous is not necessary \cite{Sch85}. We make use of these improvements in this work.

There are known results on the smallest zero of a single form in many variables. Let $\Lambda(f)$ be the smallest integer zero of a form $f\in \z[x_1,\ldots, x_n]$ with coefficients bounded in absolute value by $C$. If $d=2$, Cassels showed that
$$\Lambda(f) \leq c_n C^{\frac{n-1}{2}},$$
where the constant $c_n$ is explicit and depends only on $n$. This estimate has the best possible exponent, i.e.\ for all $n$ and $C$, there exists an $f\in \z[x_1,\ldots, x_n]$ and a constant $d_n$ only depending on $n$ such that $\Lambda(f)\geq d_nC^{\frac{n-1}{2}}$ \cite{Cas55,Cas56}. However, for generic quadratic forms one can do much better \cite{BD08}. Sardari proved an optimal strong approximation theorem for $f-N$, where $f$ is a non-degenerate quadratic form and $N$ a sufficiently large integer \cite{Sar15}. 

If $d=3$ the best possible exponent for sufficiently large $n$ is smaller than the exponent $(n-1)/2$ in Cassel's result. Browning, Dietmann and Elliot showed that $\Lambda(f)\leq c C^{360000}$ for some absolute constant $c$ provided $n\geq 17$, whereas by a result due to Pitman  one has for any $\epsilon>0$ and sufficiently large $n$ that 
$\Lambda(f)\leq c_{n,\epsilon} C^{\frac{25}{6}+\epsilon},$
for some constant $c_{n,\epsilon}$ \cite{Bro11,Pit68}. In case the hypersurface corresponding to $f$ has at most isolated ordinary singularities, the former authors provide visibly better bounds, e.g. $\Lambda(f)\leq c C^{1071}$ for $n=17$. In fact, Browning, Dietmann and Elliot wonder whether their ideas ``\emph{could be adapted to handle non-singular forms of degree exceeding $3$}'' analogous to ``\emph{the extension of Birch {\upshape\cite{Bir62}} to higher degree of Davenport's treatment of cubic forms {\upshape\cite{Dav63}}}''. The main result of this paper is that this indeed possible, although their method to achieve effective lower bounds for the singular series and integral is completely different from ours. 

\subsection{Main result}\label{sec:mr} Let $\vec{\tilde{f}}$ denote the top degree part of the system $\vec{f}$. Let $V$ and $\widetilde{V}$ denote the affine and projective variety corresponding to $\vec{f}$ and $\vec{\tilde{f}}$ respectively. Let $C$ and $\widetilde{C}$ be the (real) maximum of the absolute value of the coefficients of $\vec{f}$ and $\vec{\tilde{f}}$ respectively. 
For any $\vec{b}\in \z^r$, we let 
$\tilde{f}_{\vec{b}}=b_1\tilde{f}_1+\ldots+b_r\tilde{f}_r$. For a form $g$ we let $\Sing(g)$ be the singular locus of $g$ in affine space. Define the quantity $\Delta$ of Dietmann and Schindler and $K$ by
\begin{align}\label{eq:defK}
\Delta = \max_{\vec{b}\in \z^r\backslash\{0\}} \left(\dim\mathrm{Sing}(\tilde{f}_{\vec{b}})\right) \quad \text{and} \quad
K=\frac{n-\Delta}{2^{d-1}}.
\end{align}
Throughout this work we assume that
\begin{align}\label{eq:Birchass} K>r(r+1)(d-1) .\end{align}
In particular, $V$ is a complete intersection. Note that (\ref{eq:Birchass}) corresponds to Birch's assumption on the number of variables (i.e. equation (\ref{eq:nlarge})) after replacing the dimension of the Birch singular locus by $\Delta$. 

Our main theorem, which is proven in \cref{sec:strongapprox}, makes Birch's result, stated in the second sentence of this paper, quantitative in terms of $C$ and $\widetilde{C}$:
\begin{thm}\label{thm:main1}
Let $f_i\in \z[\vec{x}]=\z[x_1,\ldots, x_n]$ for $i=1,\ldots, r$ be polynomials of degree $d$ so that
$K-r(r+1)(d-1)>0$, $\vec{f}$ has 
 a zero over $\z_p$ for all primes $p$ and $\vec{\tilde{f}}$ has a 
 real zero. Assume that the affine and projective varieties $V$ and $\widetilde{V}$ corresponding to $\vec{f}$ and $\vec{\tilde{f}}$ respectively are non-singular. Then there exists an $\vec{x}\in \z^n\backslash\{\vec{0}\}$, polynomially bounded by $C$ and $\widetilde{C}$, such that $\vec{f}(\vec{x})=\vec{0}$, in fact
\begin{equation}\label{eq:mtestimate}
\max_{1\leq i\leq n} |x_i|\leq c (C^3\widetilde{C}^2)^{2n^2r^{n+1}(d-1)^nd\cdot\frac{K+r(r+1)(d-1)}{K-r(r+1)(d-1)}},
\end{equation}
where the constant $c$ does not depend on $C$ or $\widetilde{C}$. 
\end{thm}
The case that the system $\vec{f}$ is homogeneous is treated separately in \cref{thm:main2}. 

 \begin{remark} The bound in above theorem is in no sense believed to be optimal and far worse than the known bounds for small degrees discussed in \cref{sec:rw}. However, we provided an upper bound in a far more general setting; it was not shown that such an upper bound exists in our setting. Moreover, as pointed out in \cref{ex:exp}, in constrast to the bounds for small degrees discussed before, the exponent should grow exponentially in $n$ when $d\geq 4$ is even. The main contribution to this bound is due to our lower bound for the singular series and integral, which follows from a quantitative version of the Nullstellensatz by D'Andrea, Krick and Sombra \cite{AKS12} as discussed on \cref{p:null}. This theorem, although sharp in general, is not believed to be sharp in the present setting. It would be interesting to explore whether a stronger quantitative version of the Nullstellensatz can be applied in this setting, yielding a significant improvement in the bound (\ref{eq:mtestimate}).
 \end{remark}

\subsection{Structure of this paper}
In \cref{sec:qa} we generalise the work of Birch \cite{Bir62} to deduce an asymptotic formula (quantitative in $C$ and $\widetilde{C}$) for the number of integer points on $V$ within a box $P\mathcal{B}$ for $P\to \infty$, which is the content of \cref{thm:ass}. Familiarity with Birch's work is not necessary: we prove all results which are direct generalisations of his work. We obtain lower bounds for the singular series and integral (introduced in \cref{sec:ss} and \cref{sec:si} respectively) in \cref{sec:lbss} and \cref{sec:lbsi}. We end with a proof of our main theorems in \cref{sec:strongapprox}.

\subsection{Notation}
On the vector space $\q_p^n$ ($p$ prime or $p=\infty$) we introduce the sup norm 
$|\vec{\alpha}|_p = \max_{1\leq i\leq n}|\alpha_i|_p$, 
where  $|\cdot|_p$ is the absolute value on $\q_p$. We write $|\cdot|$ for $|\cdot|_\infty$.
For $\beta\in \r$, we let
$\|\beta\|=\min_{i\in\z}|i-\beta|$ be the least distance of $\beta$  to an integer
 and for a point $\vec{\alpha}\in \r^n$ we write 
$\|\vec{\alpha}\|=\max_{1\leq i \leq n}\|\alpha_i\|.$
If $\vec{a}\in \z^m$ and $q\in \z$, we abbreviate $\gcd(a_1,\ldots, a_m,q)$ by $\gcd(\vec{a},q).$ For $x\in \r$ and $q\in \z$ we abbreviate $e^{2\pi i x}$ by $e(x)$ and $e^{2\pi i x/q}$ by $e_q(x)$. 
For functions $f,g$ defined on a subset of the real numbers we use Vinogradov's notation $f\ll g$ to mean $f=O(g)$. Without an indication the implied constant may depend on $n, r$ and $d$, but not on $C$ or $\widetilde{C}$. 

Let $\mathcal{E}$ denote the box $[-1,1]^n$ and let $\mathcal{B}$ be an $n$-dimensional box contained in $\mathcal{E}$ of side-length at most 1, i.e. there are $a_j, b_j\in \r$ with $-1\leq a_j\leq b_j\leq 1$ and $0< b_j-a_j <1$ such that $\mathcal{B}$ is given by
$\prod_{j=1}^n [a_j,b_j].$
We abbreviate sums of the form $\displaystyle \sum_{\vec{x}\in P\mathcal{B} \cap \z^n}$ by $\displaystyle \sum_{\vec{x}\in P\mathcal{B}}$.

\section{A quantitative asymptotic formula for the number of integer zeros}\label{sec:qa}

\subsection{Estimates of exponential sums}
Let $\vec{\alpha}\in [0,1)^r$ and $\vec{\nu}\in \z^r$. We obtain estimates for exponential sums 
\begin{align}\label{def:Saq}
S(\vec{\alpha})=\sum\limits_{\vec{x}\in P\mathcal{B}} e(\vec{\alpha} \cdot \vec{f}(\vec{x})) \quad \quad \text{and} \quad \quad S(\vec{\alpha},\vec{\nu})=S(\vec{\alpha})e(-\vec{\alpha}\cdot \vec{\nu})
\end{align}
depending on $\alpha_1, \ldots, \alpha_r$ not being too well approximable by rational numbers with small denominators. 

Let $M(P, \vec{\nu})$ denote the number of integer points in the box $P\mathcal{B}$ satisfying $\vec{f}(\vec{x})=\vec{\nu}$. This counting function 
is the principal object of study in the first part of this work. 
 Observe that
\begin{align}\label{eq:Mpv}
M(P, \vec{\nu}) =\int_{\vec{\alpha} \in [0,1)^r} S(\vec{\alpha},\vec{\nu}) \dd \vec{\alpha}.
\end{align}

The following lemma, generalising \cite[Lemma 2]{Sch14}, enables us to split the right-hand side of (\ref{eq:Mpv}) into a main term and an error term. 

\begin{lem}\label{lem:3cases}
Let $\epsilon>0$ and $0<\theta<1$. One of the following holds:
\begin{enumerate}[\upshape(i)]
\item \label{it:3cases1} $\displaystyle{|S(\vec{\alpha})|\ll P^{n-K\theta+\epsilon}}$;
\item \label{it:3cases2} (rational approximation to $\vec{\alpha}$ with respect to the parameter $\theta$) there are $\vec{a}\in \z^r$ and $q\in \z_{>0}$ such that $\gcd(\vec{a},q)=1$,
\begin{align*}
2|q\vec{\alpha}-\vec{a}|< \widetilde{C}^{r-1} P^{-d+r(d-1)\theta} \quad \text{and} \quad q < \widetilde{C}^r P^{r(d-1)\theta}.
\end{align*}
\end{enumerate}
\end{lem}

\begin{proof}
\begin{sloppypar} Following the proof of Lemma 2 in \cite{Sch14}, we start by letting $\Gamma_i(\vec{x}^{(1)},\ldots, \vec{x}^{(d)})$ for $1\leq i \leq r$ be the multilinear form associated to $\tilde{f}_i$, satisfying $\Gamma_i(\vec{x},\ldots, \vec{x})=d!\tilde{f}_i(\vec{x}).$ Let $N(P^\xi,P^{-\eta};\vec{\alpha})$ be the number of integer vectors $\vec{x}^{(i)} \in \z^n$ for $2\leq i\leq d$ such that $|\vec{x}^{(i)}|\leq P^\xi$ and
$$\left\|\sum_{i=1}^r \alpha_i \Gamma_i\left(\vec{e}_j, \vec{x}^{(2)},\ldots, \vec{x}^{(d)}\right)\right\|<P^{-\eta} \quad\quad \text{for all } 1\leq j\leq n.$$
We introduce an $r \times nN(P^\theta,P^{-d+(d-1)\theta};\vec{\alpha})$ matrix $\Psi$, the rank of which enables us to distinguish between two cases. The entries of this matrix $\Psi$ are given by $\Gamma_i\left(\vec{e}_j, \vec{x}^{(2)},\ldots, \vec{x}^{(d)}\right)$ where the rows of $\Psi$ are indexed by $i$ and the columns of $\Psi$ by $\left(j, \vec{x}^{(2)},\ldots, \vec{x}^{(d)}\right)$ for $1\leq j\leq n$ and $\vec{x}^{(2)},\ldots, \vec{x}^{(d)}$ running over all vectors counted by $N(P^\theta,P^{-d+(d-1)\theta};\vec{\alpha})$. 

\underline{Case (\ref{it:3cases1}): $\mathrm{rank}\Psi < r$.} In this case, there exists a $\vec{b}\in \z^r\backslash\{\vec{0}\}$ such that $\sum\nolimits_{i=1}^r b_i \Psi_{i,l}=0$ for all $l$. In particular, the system of equations
\[\sum_{i=1}^{r} b_i \Gamma_i\left(\vec{e}_j, \vec{x}^{(2)},\ldots, \vec{x}^{(d)}\right)=0 \quad \quad (1\leq j\leq n)\]
has at least $N(P^\theta,P^{-d+(d-1)\theta};\vec{\alpha})$ integer solutions $\vec{x}^{(i)} \in \z^n$ for $2\leq i\leq d$ with $|\vec{x}^{(i)}|\leq P^\theta$. On the other hand, the number of solutions is $\ll P^{(d-2)n\theta+\Delta\theta+\epsilon}$ by the last equation on page 212 and the third equation on page 216 in \cite{Sch14}. 

Now, suppose $k$ is such that $|S(\vec{\alpha})|>P^{n-k}$. Lemma 2.4 in \cite{Bir62} implies that
$$N(P^\theta,P^{-d+(d-1)\theta};\vec{\alpha}) \gg P^{(d-1)n\theta-2^{d-1}k-\epsilon},$$
 which follows from Weyl's inequality and Davenport's application of the the geometry of numbers. This estimate is independent of the coefficients of $\vec{f}$, i.e. the implied constant does not depend on $C$ or $\widetilde{C}$. Hence, we find that 
 \[P^{(d-2)n\theta+\Delta\theta+\epsilon} \gg P^{(d-1)n\theta-2^{d-1}k-\epsilon},\]
so that  $k\geq K\theta$. Therefore,
 $$|S(\vec{\alpha})|<P^{n-K\theta+\epsilon}.$$
 
\underline{Case (\ref{it:3cases2}): $\mathrm{rank}\Psi = r$.} In this case there is an $r\times r$ submatrix of $\Psi$ of full rank, which we denote by $\hat{\Psi}$. Let $\vec{p_l}$ denote the $l$-th column of $\hat{\Psi}^{-1}$. We now define $\vec{a}$ and $q$ by 
$$q=\det \hat{\Psi}, \quad\quad \vec{a} = q\sum_{l=1}^r \vec{p_l} \left((\hat{\Psi}\vec{\alpha})_l-\big\|(\hat{\Psi}\vec{\alpha})_l\big\|\right).$$
As $\hat\Psi$ is an integer matrix and $(q\vec{p_l})_{l=1}^r$ is the adjoint of $\hat{\Psi}$, it follows that $\vec{a}\in (\z_{\geq 0})^r$ and $q\in \z$. Moreover, $q$ is non-zero as $\hat{\Psi}$ has full rank. After removing a common factor of $q$ and the $a_i$, the conditions $\gcd(q,\vec{a})=1$ and $q>0$ are satisfied.  We check that $\vec{a}$ and $q$ satisfy the remaining properties of (\ref{it:3cases2}) in the lemma.   As every entry of $\Psi$ can be estimated by $\widetilde{C}P^{(d-1)\theta}$, we find 
\begin{align} \label{eq:q} 
q\ll \widetilde{C}^rP^{r(d-1)\theta}.
\end{align} 
Moreover,
\[|q\alpha_i-a_i| = |q|\bigg|\sum_{l=1}^r(\hat{\Psi}^{-1})_{i,l}\big\|(\hat{\Psi}\vec{\alpha})_l\big\|\bigg|.\]
By Cramer's rule $q(\hat{\Psi}^{-1})_{i,l}\ll \widetilde{C}^{r-1}P^{(r-1)(d-1)\theta}$ and one has $\|(\hat{\Psi}\vec{\alpha})_l\|\ll P^{-d+(d-1)\theta}$ by construction of $\Psi$. Hence, 
\begin{align}\label{eq:approx}|q\alpha_i-a_i| \ll \widetilde{C}^{r-1}P^{(r-1)(d-1)\theta}P^{-d+(d-1)\theta}.\end{align}
Finally, note that by scaling $\theta$, the implied constants in (\ref{eq:q}) and (\ref{eq:approx}) can be transferred to the implied constant in (\ref{it:3cases1}). \qedhere
\end{sloppypar}
\end{proof} 

For $\vec{a}\in \z^r$ and $q\in \z_{>0}$,
let
\begin{align*}\label{def:Saqdiscrete}
S_{\vec{a},q}=\sum\limits_{\vec{x}  (q)} e_q(\vec{a}\cdot \vec{f}(\vec{x})) \quad \quad \text{and} \quad \quad
S_{\vec{a},q}(\vec{\nu})=S_{\vec{a},q}e_q(-\vec{a}\cdot \vec{\nu}).
\end{align*}
Here, the summation is over a complete set of residues modulo $q$ for every vector component of $\vec{x}$. The following lemma, which is a corollary of \cref{lem:3cases} and generalizes {\cite[Lemma 5.4]{Bir62}}, will be useful when we define the singular series in terms of $S_{\vec{a},q}$ in \cref{sec:ss}.

\begin{lem}\label{lem:Saq} For every $\epsilon>0$ we have
\begin{align} \label{eq:lemSaq} 
|S_{\vec{a},q}|\ll \widetilde{C}^{K/(d-1)}q^{n-K/r(d-1) + \epsilon}.
\end{align}
\end{lem}
\begin{proof}
Observe that $S_{\vec{a}',q'}$ is a particular case of $S(\vec{\alpha})$ with $P=q'$ and $\vec{\alpha}=\vec{a}'/q'$. Take $\theta$ such that $r(d-1)\theta<1-\log_q(\widetilde{C}^r).$ Then the inequalities in case (\ref{it:3cases2}) of \cref{lem:3cases} read
$$|q\vec{a}'-q'\vec{a}|<q'^{2-d}, \quad 1\leq q<q'.$$
As $d\geq 2$ this system has no solutions, hence (\ref{it:3cases1}) is satisfied. 
\end{proof}

\subsection{Minor arcs}
Let 
\[\mathcal{M}(\theta) = \{(\vec{a},q)\in \z_{\geq 0}^r \times\z_{>0} : \gcd(\vec{a},q)=1, \ |\vec{a}|\leq q < \widetilde{C}^r P^{r(d-1)\theta}\}\]
be the set of all pairs $(\vec{a},q)$ occurring in \cref{lem:3cases}(\ref{it:3cases2}). 
Given $(\vec{a},q)\in\z^r \times\z_{>0}$ and $0<\theta\leq 1$, define a \textit{major arc} by
\begin{align*}
\mathfrak{M}_{\vec{a},q}(\theta)=
\{\alpha \in [0,1)^r : 2|q\vec{\alpha}-\vec{a}|< \widetilde{C}^{r-1}P^{-d+r(d-1)\theta}\}.
\end{align*}
Then, define the \textit{major arcs} to be
\begin{align}\label{eq:majorarcs}
\mathfrak{M}(\theta)=\bigcup_{(\vec{a},q)\in \mathcal{M}(\theta)} \mathfrak{M}_{\vec{a},q}(\theta).\end{align}
Observe that $\mathfrak{M}(\theta)$ consists of all $\vec{\alpha}$ satisfying (\ref{it:3cases2}) in \cref{lem:3cases}. Define the\textit{ minor arcs} by $\mathfrak{m}=[0,1)^r\backslash\mathfrak{M}$. The contribution of $\vec{\alpha} \in \mathfrak{m}$ to the integral in (\ref{eq:Mpv}) will be considered as an error term. In order to estimate this error term, we first estimate the volume of $\mathfrak{M}(\theta)$, generalising \cite[Lemma 4.2]{Bir62}:

\begin{lem}\label{lem:vol}
The major arcs $\mathfrak{M}(\theta)$ have volume at most
\begin{align}\label{eq:maxvol} \widetilde{C}^{r^2}P^{-rd+r(r+1)(d-1)\theta}.\end{align}
\end{lem}
\begin{proof}
Each major arc $\mathfrak{M}_{\vec{a},q}(\theta)$ has volume
$\left(q^{-1}\widetilde{C}^{r-1}P^{-d+r(d-1)\theta}\right)^r.$
As $\mathfrak{M}(\theta)$ is the (not necessarily disjoint) union of major arcs, an upper bound for the volume of $\mathfrak{M}(\theta)$ is given by
\begin{align*}&\sum_{(\vec{a},q)\in \mathcal{M}(\theta)} \left(q^{-1}\widetilde{C}^{r-1}P^{-d+r(d-1)\theta}\right)^r.\qedhere\end{align*}
\end{proof}

If the major arcs are disjoint, we can write
$$\displaystyle\int_{\mathfrak{M}(\theta)} S(\vec{\alpha},\vec{\nu}) \dd \alpha = \sum_{(\vec{a},q)\in \mathcal{M}(\theta)} \displaystyle\int_{\mathfrak{M}_{\vec{a},q}(\theta)} S(\vec{\alpha},\vec{\nu}) \dd \alpha.$$
This is the case for $\theta$ small enough, generalising \cite[Lemma 4.1]{Bir62}:
\begin{lem}\label{lem:disjoint}
Suppose $\theta$ is such that $d>2r(d-1)\theta+(2r-1)\log_P(\widetilde{C})$. Then $\mathfrak{M}(\theta)$ is given as a disjoint union of $\mathfrak{M}_{\vec{a},q}(\theta)$ by \upshape(\ref{eq:majorarcs}). 
\end{lem}
\begin{proof}
Suppose that $\alpha$ lies in the distinct sets $\mathfrak{M}_{\vec{b},q}(\theta)$ and $\mathfrak{M}_{\vec{b}',q'}(\theta)$.
It follows that there is an $i$ such that $\frac{b_i}{q}\neq \frac{b_i'}{q'}.$
Then
$$1\leq |b_i'q-q'b_i|\leq q|q'\alpha_i-b_i'|+q'|q\alpha_i-b_i|<\widetilde{C}^{2r-1}P^{-d+2r(d-1)\theta},$$
which contradicts our assumption on $\theta$. 
\end{proof}

Now, take major arcs $\mathfrak{M}(\theta_0)$, where $0<\theta_0<1$, $0<\delta<1$ and $\eta$ are such that
\begin{gather}
\eta=r(d-1)\theta_0, \label{eq:defeta}\\
d>2\eta+(2r-1)\log_P(\widetilde{C}) \label{eq:uptheta} ,\\
\frac{K}{r(d-1)}-(r+1)>\delta\eta^{-1}\label{eq:delta}.
\end{gather}
Observe that assumption (\ref{eq:delta}) is a quantitative version of our main assumption (\ref{eq:Birchass}). By (\ref{eq:uptheta}) the major arcs $\mathfrak{M}_{a,q}(\theta_0)$ are disjoint. Later, we choose $\eta$ and $\delta$ satisfying (\ref{eq:uptheta}) and (\ref{eq:delta}). From now on, write $\mathfrak{M}_{\vec{a},q}$ for $\mathfrak{M}_{\vec{a},q}(\theta_0)$.

We use Birch's idea of a sliding scale to bound $S(\vec{\alpha},\vec{\nu})$ on the minor arcs. Note that the estimate
$$S(\vec{\alpha}) \ll P^{n-K\theta+\epsilon}$$
for $\alpha \in \mathfrak{m}$ is stronger the larger $\theta$ is. Therefore, in order to show that $\int_{\mathfrak{m}} |S(\vec{\alpha},\vec{\nu})| \dd \vec{\alpha}$ is negligible, we let $\theta$ depend on $\vec{\alpha}$. For most $\vec{\alpha}$, we take $\theta$ large and have a strong estimate for $|S(\vec{\alpha})|$. When this estimate is invalid, we have to use a smaller value of $\theta$, but we have the compensation that this only happens for a set of $\alpha$ of small measure by the previous lemma. So, the worse the estimate for $|S(\vec{\alpha})|$, the smaller the set of $\vec{\alpha}$ for which it is necessary to use this estimate. Hence, we find the following generalisation of \cite[Lemma 4.4]{Bir62}: 

\begin{lem}\label{lem:minor}
$$\int_{\mathfrak{m}} |S(\vec{\alpha},\vec{\nu})| \dd \vec{\alpha} = O(\widetilde{C}^{r^2}P^{n-rd-\delta}),$$
where $O$ does not depend on $C$ or $\widetilde{C}$.
\end{lem}
\begin{proof} First, observe that $|S(\vec{\alpha},\vec{\nu})|=|S(\vec{\alpha})|.$ Let $\epsilon>0$ be small. Now, define a sequence 
\begin{align}\label{eq:thetaineq}
\theta_T>\theta_{T-1}>\ldots>\theta_1>\theta=\theta_0>0
\end{align}
such that 
\begin{align}
(r+1)(d-1)\theta_T&=2d,\\
\label{eq:diftheta} r(r+1)(d-1)(\theta_{t+1}-\theta_t)&<\delta\epsilon \quad \quad\text{for } 0\leq t\leq T-1.
\end{align}
This can be done with $T\ll P^{\delta\epsilon}$ (independent of $\widetilde{C}$). 

By \cref{lem:3cases} 
and as $-K\theta_T+\epsilon < -2rd$ by (\ref{eq:Birchass}), we find
$$\int_{\vec{\alpha} \not \in \mathfrak{M}(\theta_T)} |S(\vec{\alpha},\vec{\nu})| \dd \vec{\alpha} \ll P^{n-2rd}.$$
By \cref{lem:vol} and \cref{lem:3cases}
\begin{align*}
\int_{\mathfrak{M}(\theta_{t+1})-\mathfrak{M}(\theta_t)} |S(\vec{\alpha},\vec{\nu})| \dd \vec{\alpha}  &\leq
\int_{\mathfrak{M}(\theta_{t+1})} |S(\vec{\alpha},\vec{\nu})| \dd \vec{\alpha} \\ &\ll \widetilde{C}^{r^2}P^{-rd+r(r+1)(d-1)\theta_{t+1}}P^{n-K\theta_{t}+\epsilon}.
\end{align*}
By (\ref{eq:diftheta}), (\ref{eq:delta}), (\ref{eq:defeta}) and (\ref{eq:thetaineq}) we have
\[r(r+1)(d-1)\theta_{t+1}-K\theta_{t}+\epsilon< -\delta-\delta\epsilon,\]
so that
\begin{align*}
\int_{\mathfrak{M}(\theta_{t+1})-\mathfrak{M}(\theta_t)} |S(\vec{\alpha},\vec{\nu})| \dd \vec{\alpha}
&\ll \widetilde{C}^{r^2}P^{n-rd-\delta-\delta\epsilon}.  
\end{align*}
Therefore, 
\begin{align*} 
& \int_{\alpha\not\in\mathfrak{M}(\theta_{0})} |S(\vec{\alpha},\vec{\nu})| \dd \vec{\alpha} \\ 
=&\int_{\alpha\not\in\mathfrak{M}(\theta_{T})} |S(\vec{\alpha},\vec{\nu})| \dd \vec{\alpha} +\sum_{t=0}^{T-1}\int_{\mathfrak{M}(\theta_{t+1})-\mathfrak{M}(\theta_t)}|S(\vec{\alpha},\vec{\nu})| \dd \vec{\alpha} \\
\ll & P^{n-2rd}+P^{\delta\epsilon}\widetilde{C}^{r^2}P^{n-rd-\delta-\delta\epsilon}.  \qedhere
\end{align*} 
\end{proof}

As for our choice of $\delta$ and $\theta_0$  the major arcs are disjoint, we obtain the following generalisation of \cite[Lemma 4.5]{Bir62}.
\begin{cor}[]\label{cor:minor}
$$M(P,\vec{\nu})=\sum_{(\vec{a},q)\in \mathcal{M}(\theta_0)}\int_{\mathfrak{M}_{\vec{a},q}} S(\vec{\alpha},\vec{\nu}) \dd \alpha + O(\widetilde{C}^{r^2}P^{n-rd-\delta}).$$
\end{cor}

\subsection{Approximating exponential sums by integrals}
Given $\vec{\alpha} \in \mathfrak{M}_{\vec{a},q}$, we let $\vec{\beta}=\vec{\alpha}-\vec{a}/q$. 
Similarly, given $\vec{x}\in P\mathcal{B}\cap \z^n$, we let $\vec{z}=\vec{x}-q\vec{y}$ for $\vec{y}\in\z^n$ such that $0\leq z_i<q$ for all $i$. Then, we find that
\begin{align*}
S(\vec{\alpha},\vec{\nu}) &=\sum_{\vec{z}  (q)} \sum_{\vec{z}+q\vec{y}\in P\mathcal{B} } e(\vec{\alpha}\cdot(\vec{f}(\vec{z}+q\vec{y})-\vec{\nu})) \\
&= \sum_{\vec{z}  (q)}  e_q(\vec{a}\cdot(\vec{f}(\vec{z})-\vec{\nu})) \sum_{\vec{z}+q\vec{y}\in P\mathcal{B} } e(\vec{\beta}\cdot(\vec{f}(\vec{z}+q\vec{y})-\vec{\nu})). \label{eq:presssi}
\end{align*}
The following lemma replaces the sum $$\sum_{\vec{z}+q\vec{y}\in P\mathcal{B} } e(\vec{\beta}\cdot\vec{f}(\vec{z}+q\vec{y}))$$ by an integral.  
For a measurable subset $\mathcal{C}$ of $\mathcal{E}$ and $\vec{\gamma}\in \r^r$, we write
\begin{align}\label{eq:defI} I(\mathcal{C},\vec{\gamma})=\int_{\vec{\zeta}\in \mathcal{C}} e(\vec{\gamma} \cdot \vec{\tilde{f}}(\vec{\zeta})) \dd \vec{\zeta}.\end{align}

\begin{lem}\label{lem:sumtoint} Given $\vec{z}\in \z^n, \vec{\beta}\in \z^r$ and $q\in \z_{>0}$, we have
\begin{align}\label{eq:sumtoint}\sum_{\vec{z}+q\vec{y}\in P\mathcal{B} } e(\vec{\beta}\cdot\vec{f}(\vec{z}+q\vec{y}))
=I(\mathcal{B},P^d\vec{\beta})\frac{P^n}{q^n} +O\left(\left(C|P^d\vec{\beta}|+1\right)\frac{P^{n-1}}{q^{n-1}}\right).\end{align}
\end{lem}
\begin{proof}
For the system of polynomials $\vec{r}=\vec{f}-\vec{\tilde{f}}$ of degree at most $d-1$ we have
$$|e(\vec{\beta}\cdot \vec{r}(\vec{z}+q\vec{y}))-1|\ll |\vec{\beta}||\vec{r}(\vec{z}+q\vec{y}))|\ll |\vec{\beta}| CP^{d-1},$$
where we assumed that $\vec{z}+q\vec{y}\in P\mathcal{B}$. There are $O((P/q)^n)$ values of $\vec{y}$ in the sum, hence
$$\sum_{\vec{z}+q\vec{y}\in P\mathcal{B} } e(\vec{\beta}\cdot\vec{f}(\vec{z}+q\vec{y}))=\sum_{\vec{z}+q\vec{y}\in P\mathcal{B} } e(\vec{\beta}\cdot\vec{\tilde{f}}(\vec{z}+q\vec{y}))+O\left(|\vec{\beta}|C\frac{P^{n+d-1}}{q^n}\right).$$

Next, we replace the sum in the right-hand side by the integral 
\begin{align}\label{eq:replint}\displaystyle\int_{\vec{z}+q\vec{\omega}\in P\mathcal{B}}  e(\vec{\beta}\cdot\vec{\tilde{f}}(\vec{z}+q\vec{\omega})) \dd \vec{\omega}.\end{align}
 The edges of the cube of summation and integration have length $P/q$. 
In the replacement of the sum by the integral, we have an error of at most $\ll (P/q)^{n-1}$ coming from the boundaries. The variation in $e(\vec{\beta}\cdot\vec{\tilde{f}}(\vec{z}+q\vec{y}))$ results in an error of at most
$O(|\vec{\beta}|q\widetilde{C}P^{d-1}(P/q)^n).$
Hence, the total error in (\ref{eq:sumtoint}) is 
$$\ll |\vec{\beta}| C\frac{P^{n+d-1}}{q^n}+|\vec{\beta}|\widetilde{C}\frac{P^{n+d-1}}{q^{n-1}}+\frac{P^{n-1}}{q^{n-1}} \ll \left(C|P^d\vec{\beta}|+1\right)\frac{P^{n-1}}{q^{n-1}}.$$ 
Applying the substitution $\vec{z}+q\vec{\omega}=P\vec{\zeta}$ to (\ref{eq:replint}) gives the desired result.
\end{proof}

\begin{cor}\label{cor:sumtoint}Given $\vec{z}\in \z^n$ and $\vec{\alpha} \in \mathfrak{M}_{\vec{a},q}$ with $\vec{\beta}=\vec{\alpha}-\vec{a}/q$, we have
\begin{align}\label{eq:sumtoint2}\sum_{\vec{z}+q\vec{y}\in P\mathcal{B} } e(\vec{\beta}\cdot\vec{f}(\vec{z}+q\vec{y}))
= I(\mathcal{B},P^d\vec{\beta})\frac{P^n}{q^n}+O\left(C\widetilde{C}^{r-1}\frac{P^{n+\eta-1}}{q^n}\right).\end{align}
\end{cor}

\begin{proof}
Estimate the error term in \cref{lem:sumtoint} by observing that for $\vec{\alpha}\in \mathfrak{M}_{\vec{a},q}$, it holds that $|P^d\vec{\beta}|\leq \widetilde{C}^{r-1}q^{-1}P^{\eta}$ and $1\leq \widetilde{C}^rq^{-1}P^{\eta}$.
\end{proof}

Recall $S_{\vec{a},q}(\vec{\nu})$ is defined by (\ref{def:Saqdiscrete}). Applying the corollary to (\ref{eq:presssi}) we obtain (compare with \cite[Lemma 5.1]{Bir62}):
\begin{cor}\label{cor:5.1}
Let $\vec{\alpha}=\vec{a}/q+\vec{\beta}\in \mathfrak{M}_{\vec{a},q}$. Then, 
$$S(\vec{\alpha},\vec{\nu})=q^{-n}S_{\vec{a},q}(\vec{\nu})\cdot I(\mathcal{B};P^d\vec{\beta})\cdot e(-\vec{\beta}\cdot\vec{\nu})\cdot{P^n} + O(C\widetilde{C}^{r-1}P^{n+\eta-1}).$$
\end{cor}
In the next two sections we study the singular series and singular integral which will be obtained by putting together $q^{-n} S_{\vec{a},q}(\vec{\nu})$ and $I(\mathcal{B};P^d\vec{\beta})$ respectively for all $\vec{\alpha}\in \mathfrak{M}$.

\subsection{Singular series}\label{sec:ss}
Define the \textit{singular series} by
\begin{equation*}\mathfrak{S}(\vec{\nu}) = \sum_{q=1}^\infty q^{-n} \sum_{\substack{\vec{a}  (q) \\ \gcd(\vec{a},q)=1}} S_{\vec{a},q}(\vec{\nu}).
\end{equation*}
The singular series converges absolutely under assumption (\ref{eq:delta}) on $K$. This is made quantitative in the following lemma, generalising \cite[p.256]{Bir62}:
\begin{lem}\label{lem:serconv} 
For all $\tau \geq 0$ we have
$$\sum_{q>P^{\tau\eta}} q^{-n} \sum_{\substack{\vec{a}  (q) \\ \gcd(\vec{a},q)=1}} |S_{\vec{a},q}(\vec{\nu})|\ll \widetilde{C}^{K/(d-1)} P^{-\tau\delta}.$$
\end{lem}
\begin{proof} Observe $|S_{\vec{a},q}(\vec{\nu})| = |S_{\vec{a},q}|$. By \cref{lem:Saq} and (\ref{eq:delta}), we have that 
\begin{align*}
\sum_{q>P^{\tau\eta}}q^{-n} \sum_{\substack{\vec{a}  (q) \\ \gcd(\vec{a},q)=1}}  |S_{\vec{a},q}(\vec{\nu})|
&\ll \sum_{q>P^{\tau\eta}}q^{-n} \sum_{\substack{\vec{a}  (q) \\ \gcd(\vec{a},q)=1}}  \widetilde{C}^{K/(d-1)}q^{n-K/r(d-1)+\epsilon} 
&\ll \widetilde{C}^{K/(d-1)}\sum_{q>P^{\tau\eta}}  q^{-1-\delta\eta^{-1}} \\
&\ll \widetilde{C}^{K/(d-1)} P^{-\tau\delta}. \qedhere 
\end{align*}
\end{proof}

\subsection{Singular integral}\label{sec:si}
For $T\in \r$, define a continuous function $J_T:\r^r\to \r$ by
$$J_T(\vec{\mu})=\int_{|\vec{\gamma}|\leq T} I(\mathcal{B},\vec{\gamma}) e(-\vec{\gamma} \cdot \vec{\mu}) \dd \vec{\gamma},$$
where $I(\mathcal{B};\vec{\gamma})$ is defined by \upshape{(\ref{eq:defI})}. The next two lemmata, generalising \cite[Lemma 5.2 and 5.3]{Bir62}, show that the sequence $(J_T)_{T\in\n}$ converges uniformly in $\vec{\mu}$ to a function $J(\vec{\mu})$ which we call the \textit{singular integral}.

\begin{lem}\label{lem:sch11}
For all $\epsilon>0$ one has
$$\left|I(\mathcal{B};\vec{\gamma})\right| \ll \min(1,(\widetilde{C}^{1-r}|\vec{\gamma}|)^{-r-1-\delta\eta^{-1}}(\widetilde{C}|\vec{\gamma}|)^\epsilon).$$
\end{lem}
\begin{proof}
$\left|I(\mathcal{B};\vec{\gamma})\right| \ll 1$ follows directly. 
Therefore, in proving the second part of the inequality we may assume that 
\begin{align}\label{eq:lem5.2as} \widetilde{C}^{1-r}|\gamma|>1. \end{align}
Take 
$P=\widetilde{C}|\vec{\gamma}|(\widetilde{C}^{1-r}|\vec{\gamma}|)^{K/r(d-1)}.$
By (\ref{eq:lem5.2as}) and $d\geq 2$ we find that $P>(\widetilde{C}|\vec{\gamma}|)^{2/d}$.
Hence, for $\vec{\alpha}=P^{-d}\vec{\gamma}$ we have that
$|\vec{\alpha}|<(\widetilde{C}P^d)^{-1/2}.$
Let $\phi$ be such that
$$|\vec{\alpha}| = \widetilde{C}^{r-1}P^{-d+r(d-1)\phi}.$$
Then, by \cref{lem:disjoint} we find that $\mathfrak{M}(\phi)$ is given as a disjoint union of $\mathfrak{M}_{\vec{a},q}(\phi)$. Observe that $\vec{\alpha}$ lies on the boundary of the open set $\mathfrak{M}_{\vec{0},1}(\phi)$. Hence, $\vec{\alpha}$ is not in $\mathfrak{M}(\phi)$. \Cref{lem:3cases} now implies that 
\begin{align}\label{eq:lem5.2-1}
|S(\vec{\alpha})|\ll P^{n+\epsilon}(\widetilde{C}^{1-r}P^d|\vec{\alpha}|)^{-K/r(d-1)}.
\end{align}
On the other hand, by \cref{lem:sumtoint} with $\vec{z}=\vec{0}, q=1, \vec{a}=\vec{0}$, we obtain
\begin{align}\label{eq:lem5.2-2}
S(\vec{\alpha}) = \sum_{\vec{y}\in P\mathcal{B} } e(\vec{\alpha}\cdot \vec{f}(\vec{y})) =  I(\mathcal{B},P^d\vec{\alpha})P^n + O\left(\left(\widetilde{C}|P^d\vec{\alpha}|+1\right)P^{n-1}\right).
\end{align}
Hence, combining (\ref{eq:lem5.2-1}) and (\ref{eq:lem5.2-2}) yields
$$\left| I(\mathcal{B},\vec{\gamma})\right| \ll (\widetilde{C}^{1-r}|\vec{\gamma}|)^{-K/r(d-1)}(\widetilde{C}|\vec{\gamma}|)^\epsilon.$$
Estimating $K/r(d-1)$ by $r+1+\delta\eta^{-1}$ using (\ref{eq:delta}) completes the proof.
\end{proof}

\begin{lem}\label{lem:Jerror} If $T_2>T_1$, for all $\epsilon>0$ one has
\begin{align}\label{eq:Jerror} |J_{T_1}(\vec{\mu})-J_{T_2}(\vec{\mu})|\ll \widetilde{C}^{r^2-1+(r-1)\delta\eta^{-1}+\epsilon} T_1^{-1-\delta\eta^{-1}+\epsilon}.\end{align}\end{lem}
\begin{proof}
Using \cref{lem:sch11} we find
\begin{align*}
J_{T_2}(\vec{\mu})-J_{T_1}(\vec{\mu}) 
&= \int_{T_1 \leq |\vec{\gamma}| \leq T_2} I(\mathcal{B},\vec{\gamma})e(-\vec{\gamma} \cdot \vec{\mu}) \dd\vec{\gamma} \\
&\ll \int_{T_1 \leq |\vec{\gamma}| \leq T_2} (\widetilde{C}^{1-r}|\vec{\gamma}|)^{-r-1-\delta\eta^{-1}}(\widetilde{C}|\vec{\gamma}|)^\epsilon \dd\vec{\gamma} \\
&\ll \int_{T_1}^{T_2} \widetilde{C}^{r^2-1+(r-1)\delta\eta^{-1}+\epsilon}\Gamma^{-r-1-\delta\eta^{-1}+\epsilon} \Gamma^{r-1}\dd\Gamma \\
&\ll \widetilde{C}^{r^2-1+(r-1)\delta\eta^{-1}+\epsilon} T_1^{-1-\delta\eta^{-1}+\epsilon}.\qedhere
\end{align*}
\end{proof}
Taking the limit $T_2\to \infty$ we obtain
\begin{equation}\label{eq:Jerror2} J_{T_1}(\vec{\mu})-J(\vec{\mu})\ll \widetilde{C}^{r^2-1+(r-1)\delta\eta^{-1}+\epsilon} T_1^{-1-\delta\eta^{-1}+\epsilon}.\end{equation}
This implies the following upper bound for $J(\vec{\mu})$:

\begin{cor}\label{cor:upJint} For all $\vec{\mu}\in \z^r$ and $\epsilon>0$ it holds that
$$J(\vec{\mu})\ll \widetilde{C}^{r(r-1)+\epsilon}.$$
\end{cor}
\begin{proof}
By the trivial bound in \cref{lem:sch11} we have $J_{\widetilde{C}^{r-1}}(\vec{\mu})
\ll \widetilde{C}^{r(r-1)}.$
By (\ref{eq:Jerror2}) we have
$ J(\vec{\mu})-J_{\widetilde{C}^{r-1}}(\vec{\mu}) \ll 
\widetilde{C}^{r(r-1)+\epsilon}.$ The result follows by the triangle inequality. 
\end{proof}

\subsection{Major arcs }
Combining the results in the previous sections, we obtain a quantitative asymptotic theorem for the number of integer zeros of $\vec{f}-\vec{v}$ in a box $P\mathcal{B}$, generalising \cite[Lemma 5.5]{Bir62}:
\begin{lem}\label{lem:ass1}
$$\frac{M(P,\vec{\nu})}{P^{n-rd}} =  \mathfrak{S}(\vec{\nu}) J(P^{-d}\vec{\nu})+\mathfrak{O}_1+\mathfrak{O}_2,$$
where
$$\mathfrak{O}_1= O\left(C\widetilde{C}^{2r^2+r-1}P^{-1+2(r+1)\eta}\right) \quad \text{and} \quad \mathfrak{O}_2=O\left(\widetilde{C}^{K/(d-1)+r^2-r+\epsilon}P^{-\delta}\right).$$
\end{lem}

\begin{proof}
By \cref{cor:minor} we have that
\[M(P,\vec{\nu}) = \sum_{(\vec{a},q)\in \mathcal{M}(\theta_0)}
 \displaystyle\int_{\mathfrak{M}_{\vec{a},q}} S(\vec{\alpha},\vec{\nu}) \dd \vec{\alpha}+O(\widetilde{C}^{r^2}P^{n-rd-\delta}).\]
Note that $K/(d-1)-r+\epsilon>0$. Let $\vec{\beta}=\vec{\alpha}-\vec{a}/q$. Then
 \[M(P,\vec{\nu})
=\sum_{(\vec{a},q)\in \mathcal{M}(\theta_0)}
 \displaystyle\int_{\mathfrak{M}_{\vec{0},q}} S(\vec{a}/q+\vec{\beta},\vec{\nu}) \dd \vec{\beta}+P^{n-rd}\mathfrak{O}_2. \]
As $|S_{\vec{a},q}(\vec{\nu})|\leq q^n$ and $|\mathcal{M}(\theta_0)|\leq (\widetilde{C}^r P^\eta)^{r+1}$, 
\[\sum_{(\vec{a},q)\in \mathcal{M}(\theta_0)} \frac{|S_{\vec{a},q}(\vec{\nu})|}{q^n} C\widetilde{C}^{r-1}P^{n+\eta-1}\int_{\mathfrak{M}_{\vec{0},q}}\dd\vec{\beta} = 
 P^{n-rd}\mathfrak{O}_1.  \]
Hence, \cref{cor:5.1} implies that
\begin{align}\label{eq:prop214-1}
\frac{M(P,\vec{\nu})}{P^{n-rd}}
=P^{rd}\sum_{(\vec{a},q)\in \mathcal{M}(\theta_0)} & \frac{ S_{\vec{a},q}(\vec{\nu})}{q^n} J_{\widetilde{C}^{r-1}P^{\eta}}(P^{-d}\vec{\nu}) + \mathfrak{O}_1+\mathfrak{O}_2.
\end{align}
By \cref{lem:serconv} for $\tau=0$  we find
\begin{align*}
\sum_{(\vec{a},q)\in \mathcal{M}(\theta_0)}  & \frac{|S_{\vec{a},q}(\vec{\nu})|}{q^n}\widetilde{C}^{r^2-1+(r-1)\delta\eta^{-1}+\epsilon}(\widetilde{C}^{r-1}P^\eta)^{-1-\delta\eta^{-1}+\epsilon}=\mathfrak{O}_2.
\end{align*}
Hence, using (\ref{eq:Jerror2}) for $T_1=\widetilde{C}^{r-1}P^{\eta}$ to rewrite (\ref{eq:prop214-1}) we obtain
\begin{align*}
\frac{M(P,\vec{\nu})}{P^{n-rd}}
= \sum_{(\vec{a},q)\in \mathcal{M}(\theta_0)} & \frac{ S_{\vec{a},q}(\vec{\nu})}{q^n} J(P^{-d}\vec{\nu}) +\mathfrak{O}_1+\mathfrak{O}_2.
\end{align*}
By \cref{lem:serconv} for $\tau=1$ and \cref{cor:upJint} we can plug in the singular series and obtain	
\begin{align*}  
\frac{M(P,\vec{\nu})}{P^{n-rd}}
&= \left(\mathfrak{S}(\vec{\nu})+O(\widetilde{C}^{K/(d-1)}P^{-\delta})\right) J(P^{-d}\vec{\nu})+\mathfrak{O}_1+\mathfrak{O}_2\\
&= \mathfrak{S}(\vec{\nu})J(P^{-d}\vec{\nu})+\mathfrak{O}_1+\mathfrak{O}_2. \qedhere
\end{align*}
\end{proof}

\begin{thm}\label{thm:ass}
\begin{align*}
M(P,\vec{\nu})=P^{n-rd}\mathfrak{S}(\vec{\nu})J(P^{-d}\vec{\nu})+O(C\widetilde{C}^{K/(d-1)+r^2-1}P^{n-rd-\delta}),
\end{align*}
where 
\begin{equation}\label{eq:defdelta}\delta < \frac{K-r(r+1)(d-1)}{K+r(r+1)(d-1)}.\end{equation}
\end{thm}
\begin{proof}
Let $\delta$ as in (\ref{eq:defdelta}) and let
$$\eta=\frac{r(d-1)}{K+r(r+1)(d-1)}.$$
As $\eta<\frac{1}{r+1}$, equation (\ref{eq:uptheta}) is satisfied for $P$ large enough (e.g. $P>\widetilde{C}^{\frac{4r}{d}}$). Moreover, (\ref{eq:delta}) is satisfied. As $-1+2(r+1)\eta<-\delta$ and $r^2+r<K/(d-1)$, both error terms $\mathfrak{O}_1$ and $\mathfrak{O}_2$ in \cref{lem:ass1} are $O(C\widetilde{C}^{K/(d-1)+r^2-1}P^{-\delta})$.  The statement now follows from that result.
\end{proof}

\section{Quantitative strong approximation}\label{chap:qsa}

\subsection{The Nullstellensatz for non-singular varieties}\label{p:null} 
From now on assume that $V$ and $\widetilde{V}$ are non-singular over $\overline{\q}$ as affine and projective varieties respectively (see \cref{sec:mr} for both our main assumption (\ref{eq:Birchass})\footnote{In fact, for the lower bounds for the singular series and singular integral it suffices to assume $n\geq r$, but we need (\ref{eq:Birchass}) again in \cref{sec:strongapprox}.} and the definition of $V$ and $\widetilde{V}$). Then, if $\vec{f}$ has a zero modulo a prime power, we can invoke Hensel's lemma to find more zeros modulo higher powers of the same prime. The real analogue of this statement is the implicit function theorem around a zero of $\vec{\tilde{f}}$. These ideas can be used to deduce known results on the non-vanishing of the singular series and integral provided the existence of local zeros. In \cref{prop:s0} and \cref{cor:J0} we prove this in a quantitative manner. In order to do so, we need to control the minors of Jacobian matrix of $\vec{f}$ and $\vec{\tilde{f}}$, for which we use a quantitative version of the Nullstellensatz, as explained below. 

\label{def:delta} Let $I$ be a subset of $[n]:=\{1,\ldots, n\}$ of size $r$ and let $\Delta_I(\vec{x})$ be the $r\times r$ minor of the Jacobian matrix of $\vec{f}$ (of dimensions $r\times n$) with columns given by the elements of $I$. 
Similarly, let $\tilde{\Delta}_I(\vec{x})$ be the $r\times r$ minor of the Jacobian matrix of $\vec{\tilde{f}}$  with columns given by the elements of $I$. 

Consider the polynomials $\vec{f}$ and all $r\times r$ minors $\Delta_I$. As $V$ is non-singular, these polynomials have no common zero over $\overline{\q}$. Hence, by the Nullstellensatz, the ideal generated by these polynomials equals $\overline{\q}[\vec{x}]$. This is made quantitative in \cite[Theorem 0.2]{AKS12} (we take $V=\mathbb{A}^n(\overline{\q}), g=1$ and $s=r+\binom{n}{r}$ in this result): there exists an $N\in \z_{>0}$ and polynomials $g_1,\ldots, g_r$ and $g_I$ in $\z[\vec{x}]$ for all $I\subset [n]$ with $|I|=r$ such that
\begin{align}\label{eq:defN}\sum_{i=1}^r f_i(\vec{x}) g_i(\vec{x}) + \sum_I \Delta_I(\vec{x}) g_I(\vec{x}) =N, \end{align}
satisfying the estimate
$\log(N)\ll 2(n+1)r^n(d-1)^{n-1}d\log(C^r).$
Hence,
\begin{align}\label{eq:defc}N\ll C^{2(n+1)r^{n+1}(d-1)^{n-1}d}=\mathfrak{C},\end{align}
where the above equation defines $\mathfrak{C}$. 

For the projective variety $\widetilde{V}$ we have a similar reasoning for every affine patch of $\widetilde{V}$ obtained by setting one of the coordinates $x_j$ equal to $1$. Let $1\leq j \leq n$ be given. Because $\widetilde{V}$ is non-singular over $\overline{\q}$, we find $\tilde{N}_j\in \z_{>0}$ and polynomials $\tilde{g}_{1,j},\ldots, \tilde{g}_{r,j}$ and $\tilde{g}_{I,j}$ in $\z[\vec{x}]$ for all $I\subset [n]$ with $|I|=r$ such that
\begin{align}\label{eq:defNtilde}\sum_{i=1}^r \tilde{f}_i(\vec{x}) \tilde{g}_{i,j}(\vec{x}) + \sum_I \tilde{\Delta}_I(\vec{x}) \tilde{g}_{I,j}(\vec{x}) =\tilde{N}_j \end{align}
for all $\vec{x}$ with $x_j=1$.  
Let $\|g\|_{\infty}$ denote the height of a polynomial $g$, that is $\|g\|_{\infty}$ is the maximum of the absolute values of the coefficients of $g$. Then, by the same result \cite[Theorem 0.2]{AKS12} we have
\begin{align*}
\log\|\tilde{g}_{I,j}\|_{\infty} &\ll 2nr^{n-1}(d-1)^{n-2}d\log(\widetilde{C}^r)
\end{align*}
for all $I\subset[n]$ with $|I|=r$. Hence,
\begin{align}\label{eq:defctilde}\|\tilde{g}_{I,j}\|_{\infty} \ll \widetilde{C}^{2nr^{n}(d-1)^{n-2}d}=\tilde{\mathfrak{C}},\end{align}
where the above equation defines $\tilde{\mathfrak{C}}$. Let $\tilde{N}=\min_{j=1}^n \tilde{N}_j$.

\subsection{Lower bound for the singular series }\label{sec:lbss}
As usual, for each prime $p$ define the \textit{local density} at $p$ to be
\begin{equation}\label{eq:defld}\sigma_p(\vec{\nu})=\lim_{m\to \infty}\frac{\#\{\vec{x} \Mod p^m \mid \vec{f}(\vec{x})\equiv \vec{\nu} \Mod p^m\}}{p^{m(n-r)}}.\end{equation}
Observe that
\[p^{m(n-r)}\sum_{k=0}^m \sum_{\substack{\vec{a} (p^k) \\ \gcd(\vec{a},p)=1}} p^{-kn} S_{\vec{a},p^k}(\vec{\nu})\]
is the number of points satisfying $\vec{f}(\vec{x})=\vec{\nu} \mod p^m$. So, equivalently we could have defined $\sigma_p(\vec{\nu})$ by 
$$\sigma_p(\vec{\nu})=\sum_{k=0}^\infty \sum_{\substack{\vec{a}  (p^k) \\ \gcd(\vec{a},p)=1}} p^{-kn} S_{\vec{a},p^k}(\vec{\nu}).$$
Then, by multiplicativity of $S_{\vec{a},q}$ we can factorize the singular series as a product over the local densities, i.e.
$\mathfrak{S}(\vec{\nu}) = \prod_{p \text{ prime}} \sigma_p(\vec{\nu}).$ The rest of this section is devoted to quantitative lower bounds for the local densities and singular series, using the ideas described in the previous section. 

\begin{lem}\label{lem:lowld} If there exists a non-singular solution $\vec{x}_0\in \z_p^n$ to $\vec{f}(\vec{x}_0)=\vec{\nu}$, then
$$\sigma_p(\vec{\nu})\geq (p^{-1}\max_I|\Delta_I(\vec{x}_0)|_p^2)^{n-r}.$$
 \end{lem}

\begin{proof}
Take $e\in \z$ such that $\max_I|\Delta_I(\vec{x}_0)|_p=p^{-e}$ and assume that $m>2e+1$. The non-singular solution $\vec{x}_0\in \z_p^n$ gives a non-singular solution $\vec{a}$ modulo $p^{2e+1}$. Using Hensel's lemma (see, for example, \cite[Proposition 5.21 and Note 5.22]{Gre69}), we can lift this solution to at least $p^{(n-r)(m-2e-1)}$ non-singular solutions of $\vec{f}(\vec{x})\equiv \vec{\nu} \mod p^m$. Hence,
by (\ref{eq:defld}) the result follows. 
\end{proof}

\begin{lem}\label{lem:lowdisc}
For all primes $p$ for which there exists a solution $\vec{x}_0\in \z_p^n$ of $\vec{f}(\vec{x}_0)=\vec{0}$ we have
$$\max_I|\Delta_I(\vec{x}_0)|_p \geq |N|_p.$$
\end{lem}
\begin{proof}
Let $p$ be a prime such that there exists an $\vec{x}_0\in \z_p^n$ with $\vec{f}(\vec{x}_0)=\vec{0}$, so that the first set of terms on the left-hand side of (\ref{eq:defN}) vanish for $\vec{x}=\vec{x}_0$. Then taking $p$-adic absolute values in (\ref{eq:defN}) shows that
$$\max_I|\Delta_I(\vec{x}_0)|_p \max_I |g_I(\vec{x}_0)|_p \geq |N|_p.$$
As $g_I\in \z[\vec{x}]$ we obtain $\max_I|\Delta_I(\vec{x}_0)|_p \geq |N|_p.$
\end{proof}

\begin{lem}\label{lem:goodprime} If $p$ is prime such that $p\nmid d$ and $p\nmid N$, then 
\begin{align}\label{eq:lemgoodprime}\sigma_p(\vec{0})-1\ll p^{-n/2+r+\epsilon}.\end{align}
\end{lem}
\begin{proof}
Suppose $V$ is singular over $\f_p$. Then, there exists an $\vec{x}\in \f_p^n$ such that $\vec{f}(\vec{x})=\vec{0}$ and $\Delta_I(\vec{x})=0$ over $\f_p$ for all $I\subset [n]$ with $|I|=r$. Considering (\ref{eq:defN}) over $\f_p$, it follows that $N\equiv 0 \mod p$. This contradicts our assumption, so $V$ is non-singular over $\f_p$. 

As pointed out by Schmidt \cite{Sch84}, a result of Deligne, worked out in the appendix of \cite{Ser77}, then shows that
$$\# V_{\f_p}(\vec{0}) = p^{n-r}+O(p^{n/2+\epsilon})$$
provided $p\nmid d$, where the implied constant depends at most on $n$ and $d$. Observe that if $\vec{x}\in \z^n$ is a solution of $\vec{f}(\vec{x})=\vec{0} \mod p^e$ for some $e\in \z_{>0}$, then $\vec{x}$ reduces to a non-singular point on $V_{\f_p}$. Hence, $\vec{x} \mod p^e$ can be obtained by lifting a point of $V_{\f_p}$. We conclude that
\begin{align*}
&\#\{\vec{x} \Mod p^m \mid \vec{f}(\vec{x})\equiv \vec{\nu} \Mod p^m\} = p^{m(n-r)}+O(p^{(m-1)(n-r)+n/2+\epsilon}).
\end{align*}
Using (\ref{eq:defld}), we obtain (\ref{eq:lemgoodprime}).
\end{proof}

\begin{prop}\label{prop:s0}
Suppose that for all primes $p$ there exists a solution $\vec{x}_0\in \z_p^n$ to $\vec{f}(\vec{x}_0)=\vec{0}$. Then
$$\mathfrak{S}(\vec{0}) \gg N^{-3(n-r)}.$$
\end{prop}
\begin{proof}
Let $S$ be the finite set of primes for which $p\mid dN$. Applying \cref{lem:lowld} and \cref{lem:lowdisc} and using the product formula for $|\cdot|_p$ we obtain
\begin{align}\label{eq:lemss1}\prod_{p\in S} \sigma_p(\vec{0}) \geq \prod_{ p\in S}(p^{-1}|N|_p^2)^{n-r}\gg (N^{-1}N^{-2})^{n-r}=N^{3(r-n)}.\end{align}
It follows from \cref{lem:goodprime} that
$$\prod_{p\not \in S} \sigma_p(\vec{0})= \prod_{p\not \in S} 1+O(p^{-n/2+r+\epsilon})\gg 1,$$
where the implied constant does not depend on $C$. 
Therefore, we conclude that
\begin{align*}
\mathfrak{S}(\vec{0}) 
&= \prod_{p\in S} \sigma_p(\vec{0})\prod_{p\not \in S} \sigma_p(\vec{0}) 
\gg N^{-3(n-r)}. \qedhere
\end{align*} 
\end{proof}

\subsection{Lower bound for the singular integral}\label{sec:lbsi}
The following lemma is a quantitative version of the implicit function theorem. We use this result to prove \cref{lem:lowint}, which is the real analogue of \cref{lem:lowld}. Recall that $\tilde{\Delta}_I(\vec{x})$ is the $r\times r$ minor of the Jacobian of $\tilde{f}$ with columns determined by $I$.
Abbreviate $\tilde{\Delta}_{\{1,2,\ldots, r\}}(\vec{x})$ by $\tilde{\Delta}(\vec{x})$. 

\begin{lem}\label{lem:ift}\begin{sloppypar} Given $\vec{x}_0\in \r^n$ with $|\vec{x}_0|\leq 1$, assume that 
$$M:=\max\limits_{I\subset [n], |I|=r} |\tilde{\Delta}_I(\vec{x}_0)|=|\tilde{\Delta}(\vec{x}_0)| >0.$$ 
Let $\vec{g}: \r^n\to \r^n$ be given by
$$\vec{g}: \vec{x}\mapsto (\tilde{f}_1(\vec{x}),\ldots, \tilde{f}_r(\vec{x}), x_{r+1}, \ldots, x_n).$$
Then there are open subsets $U\subset \r^n$ and $W\subset \r^n$ with $\vec{x}_0\in U$ and $\vec{g}(\vec{x}_0)\in W$ such that $\vec{g}$ is a bijection from $U$ to $W$ and has differentiable inverse $\vec{g}^{-1}:W\to U$ satisfying $\det((\vec{g}^{-1})')\geq M^{-1}$. Furthermore, one may choose
$$W=\left\{\vec{y}\in \r^n : |\vec{g}(\vec{x}_0)-\vec{y}| < \frac{M^2}{\widetilde{C}^{2r-1}}\right\}.$$
\end{sloppypar}
\end{lem}

\begin{proof} We explicitly find a small open neighbourhood of $\vec{x}_0$ in which the implicit function theorem is applicable, following the proof of \cite[Theorem 2.11]{Spi65} or \cite[Lemma 9.3]{PSW16}. 
Note that 
$\label{eq:M}M\ll \widetilde{C}^r.$ 
Let $U$ be the closed rectangle given by
\begin{align}\label{def:U} U=\left\{\vec{x}\in \r^n: |\vec{x}-\vec{x}_0|\leq a \frac{M}{\widetilde{C}^{r}}\right\},\end{align}
for a sufficiently small constant $a\in \r$ depending only on $d,n$ and $r$. Then for $\vec{x} \in U$ we have 
that
$|\vec{x}|\leq |\vec{x}-\vec{x}_0|+|\vec{x}_0| 
 \ll 1.$
Hence, for $\vec{x}\in U$ one finds that $\pdv{g_i}{x_j x_k}(\vec{x}) \ll \widetilde{C}$ for all $1\leq i,j,k \leq n$. 
It follows that
\begin{align}\label{eq:lem4.5prop1}\left|\pdv{g_i}{x_j}(\vec{x}) - \pdv{g_i}{x_j}(\vec{x}_0)\right|\ll \widetilde{C}|\vec{x}-\vec{x}_0|
\ll a\widetilde{C}^{1-r}M.\end{align}
Let $D\vec{g}$ be the Jacobian matrix of $\vec{g}$ and write $\vec{g}'(\vec{x})=\vec{g}(\vec{x})-D\vec{g}(\vec{x}_0)\cdot \vec{x}$. As 
$$\pdv{\left(\vec{g}'(\vec{x})\right)_i}{x_j}=\pdv{g_i}{x_j}(\vec{x}) - \pdv{g_i}{x_j}(\vec{x}_0),$$ for $\vec{x}_1, \vec{x}_2\in U$ we have
\begin{align} \label{eq:liminv4} 
|\vec{g}'(\vec{x}_1)-\vec{g}'(\vec{x}_2)| \ll a\widetilde{C}^{1-r}M|\vec{x}_1-\vec{x}_2|.
\end{align}

Given an invertible $n\times n$-matrix $A$, let $|A|=\max_{i,j} |A_{i,j}|$ the max norm. For all $\vec{h}\in \r^n$ one has
$|\vec{h}|
\leq \frac{|\mathrm{adj}(A)|}{\det A}|A\vec{h}|$
with $\mathrm{adj}(A)$ the adjugate of $A$. 
Let $A=D\vec{g}(\vec{x}_0)$. 
Then $|\mathrm{adj}(A)|\ll \widetilde{C}^{r-1}$. 
Since $M=|\tilde{\Delta}(\vec{x}_0)|= |\det(D\vec{g}(\vec{x}_0))|$, we find that for $\vec{x}_1, \vec{x}_2\in U$ we have
$$|D\vec{g}(\vec{x}_0)(\vec{x}_1-\vec{x}_2)|\gg 
\widetilde{C}^{1-r}M|\vec{x}_1-\vec{x}_2|.$$
Hence, 
\begin{align*}
|\vec{g}'(\vec{x}_1)-\vec{g}'(\vec{x}_2)|+|\vec{g}(\vec{x}_1)-\vec{g}(\vec{x}_2)| 
&\geq  |D\vec{g}(\vec{x}_0)(\vec{x}_1-\vec{x}_2)|\\
&\gg  \widetilde{C}^{1-r} M|\vec{x}_1-\vec{x}_2|.\end{align*}
Therefore, using (\ref{eq:liminv4}) for $a$ small enough, we find for all $\vec{x}_1,\vec{x}_2\in U$ that
$$|\vec{g}(\vec{x}_1)-\vec{g}(\vec{x}_2)|\gg \widetilde{C}^{1-r} M|\vec{x}_1-\vec{x}_2|.$$
This implies that if $\vec{x}$ is on the boundary of $U$ we have
\begin{align}\label{eq:boundary}|\vec{g}(\vec{x})-\vec{g}(\vec{x}_0)|\gg  \widetilde{C}^{1-r}M |\vec{x}-\vec{x}_0|=a\frac{M^2}{\widetilde{C}^{2r-1}}\end{align}
Set $b\gg a\frac{M^2}{\widetilde{C}^{2r-1}}$ so that for $\vec{x}$ on the boundary of $U$ it holds that $|\vec{g}(\vec{x})-\vec{g}(\vec{x}_0)|\gg b$ and define
$$W=\{\vec{y}\in \r^n : |\vec{y}-\vec{g}(\vec{x}_0)|<\tfrac{1}{2}b\}.$$
The proof of \cite[Lemma 9.3]{PSW16} ensures that $W$ has the required properties (after shrinking $U$). 
\end{proof}

\begin{lem}\label{lem:lowint}
Suppose that $\vec{x}_0\in \r^n$ with $|\vec{x}_0|\leq 1$ satisfies $\vec{\tilde{f}}(\vec{x}_0)=\vec{0}$ such that $M=\max\limits_{I\subset [n], |I|=r} |\tilde{\Delta}_I(\vec{x}_0)|>0$. Then, we have
$$J(\vec{0})\gg M^{-1}\left(\frac{M^2}{\widetilde{C}^{2r-1}}\right)^{n-r}.$$
\end{lem}
\begin{proof}
In \cite[Paragraph 11]{Sch82}, Schmidt shows that for $\vec{\mu}\in \r^r$ we have
$$J(\vec{\mu}) = \lim_{t \to \infty} t^r \int_{|\vec{\tilde{f}}(\vec{x})-\vec{\mu}|\leq t^{-1}}\prod_{i=1}^r (1-t|\tilde{f}_i(\vec{x})-\mu_i|) \dd \vec{x}.$$
Let $\one_{1/2t}:\r\to \{0,1\}$ be the characteristic function of the interval $[-\frac{1}{2t},\frac{1}{2t}]$. Let $U,W,g$ as in \cref{lem:ift}. Then,
$$J(\vec{0}) \geq \lim_{t \to \infty} \left(\frac{t}{2}\right)^r \int_U \prod_{i=1}^r\one_{1/2t}\circ\tilde{f}_i(\vec{x}) \dd \vec{x}.$$
Applying the change of variables as in \cref{lem:ift} we obtain
\begin{align*}
\int_U \prod_{i=1}^r\one_{1/2t}\circ\tilde{f}_i(\vec{x}) \dd \vec{x} 
&= \int_W |\det((g^{-1})')|\prod_{i=1}^r\one_{1/2t}(y_i) \dd \vec{y} \\
&\geq \int_W M^{-1}\prod_{i=1}^r\one_{1/2t}(y_i) \dd \vec{y}. 
\end{align*}
As for $t$ sufficiently large $\one_{1/2t}\equiv 0$ outside $W$, the theorem follows.
\end{proof}

\begin{lem}\label{lem:lowdiscr} Let $\vec{x}_0\in \r^n$ be such that $|\vec{x}_0|=1$ and $\vec{\tilde{f}}(\vec{x}_0)=\vec{0}$. Then for some $1\leq j\leq n$ one has
$$\max_I |\tilde{\Delta}_I(\vec{x}_0)| \gg \tilde{\mathfrak{C}}^{-1}\tilde{N}_j.$$ 
\end{lem}
\begin{proof}
This is essentially the same proof as the proof of \cref{lem:lowdisc}. Substitute $\vec{x}=\vec{x}_0$ in (\ref{eq:defNtilde}) for a choice of $j$ such that $(x_0)_j=|\vec{x}_0|=1$. Then the first sum vanishes and we find that
$$\max_I|\tilde{\Delta}_I(\vec{x}_0)|\max_I|\tilde{g}_{I,j}(\vec{x}_0)| \gg |\tilde{N}_j|.$$
Note that $\tilde{g}_{I,j}(\vec{x}_0)\ll \|\tilde{g}_{I,j} \|_\infty \ll \mathfrak{\widetilde{C}}$. This implies that
\begin{align}
\max_I|\tilde{\Delta}_I(\vec{x}_0)|&\gg \mathfrak{\widetilde{C}}^{-1}\tilde{N}_j.  \qedhere 
\end{align}
\end{proof} 

\begin{cor}\label{cor:J0}
Suppose $\vec{\tilde{f}}$ has a real zero. Then
$$J(\vec{0})\gg \frac{\tilde{N}^{2(n-r)-1}}{\tilde{\mathfrak{C}}^{2(n-r)-1}\widetilde{C}^{(2r-1)(n-r)}}.$$
\end{cor}
\begin{proof} Observe that by homogeneity of $\vec{\tilde{f}}$ we can assume that the non-singular real zero $\vec{x}_0$ satisfies $|\vec{x}_0|=1$. The corollary then follows directly from \cref{lem:lowint} and \cref{lem:lowdiscr}. \end{proof}

\subsection{Main theorems}\label{sec:strongapprox}
\begin{proof}[Proof of \cref{thm:main1}]
From \cref{thm:ass}, it follows that for $P$ satisfying
$$P\gg \left(\frac{C\widetilde{C}^{K/(d-1)+r^2-1}}{\mathfrak{S}(\vec{0})J(\vec{0})}\right)^{1/\delta}$$
we have that $M(P,\vec{0})> 1$ (if the implied constant is large enough). Hence, there exists a non-trivial integer zero $\vec{x}$ of $\vec{f}$ with $|\vec{x}| \leq P$. 

By \cref{prop:s0}, \cref{cor:J0}, (\ref{eq:defc}) and $\tilde{N}\geq 1$ it follows that
\begin{align*}
\mathfrak{S}(\vec{0})J(\vec{0}) &\gg \tilde{\mathfrak{C}}^{-(2(n-r)-1)}\widetilde{C}^{-2r(n-r)}\left(\frac{\tilde{N}^2}{N^3}\right)^{n-r}\tilde{N}^{-1}\\
&\gg \mathfrak{C}^{-3(n-r)}\tilde{\mathfrak{C}}^{-(2(n-r)-1)}\widetilde{C}^{-(2r-1)(n-r)}.
\end{align*}
Using that $(n+1)(n-r)< n^2-r$ one finds
\begin{align*}
\frac{C\widetilde{C}^{K/(d-1)+r^2-1}}{\mathfrak{S}(\vec{0})J(\vec{0})}\ll \mathfrak{C}^{\frac{3n^2}{n+1}}\tilde{\mathfrak{C}}^{\frac{2n^2}{n+1}}\frac{C}{\mathfrak{C}^{\frac{r}{n+1}}}\frac{\widetilde{C}^{K/(d-1)+(2r-1)(n-r)+r^2-1}}{\mathfrak{\widetilde C}^{1+\frac{r}{n+1}}}
\end{align*}
As $K\leq \frac{n}{2^{d-1}}$, the two fractions on the right-hand side are bounded by $1$, so that one can take
\begin{align*}
P&\gg (C^3\widetilde{C}^2)^{2n^2r^{n+1}(d-1)^nd\cdot\frac{K+r(r+1)(d-1)}{K-r(r+1)(d-1)}}.\qedhere
\end{align*}
\end{proof}

\begin{remark} \label{ex:exp}
The upper bound (\ref{eq:mtestimate}) should be compared with the following example generalising Kneser's example in \cite{Cas56}. Suppose $d\geq 4$ is even and let 
$$f(\vec{x})=x_1^d-\sum_{i=1}^{n-1}(x_{i+1}-cx_i^{d/2})^2-1$$
for some $c\in \n$. Then $C=\widetilde{C}=c^2$, $f$ is non-singular and $\vec{x}$ given by $x_i=c^{\sum_{j=1}^{i-1}(d/2)^{i-1}}$ is a zero of $f$. If $\vec{a}$ is a non-trivial integer zero of $\vec{f}$, then clearly $a_1\neq 0$. Moreover,
$$|a_{i+1}-ca_i^{d/2}|\leq |a_1|^{d/2}$$
for all $i=1,2,\ldots, n-1$. Inductively one can show that
$$|a_{i}|\gg c^{\left(d/2\right)^{i-2}}|a_1| \quad \quad (2\leq i \leq n),$$
where the implied constant only depends on $i$. Hence, in case $r=1$ and $d\geq 4$ is even, the right-hand side of (\ref{eq:mtestimate}) is at least $\displaystyle C^{\tfrac{1}{2}(d/2)^{n-2}}$. Note that---in contrast to the cases $d=2$ and $d=3$---the exponent of $C$ grows exponentially in $n$. 

\end{remark}

We can do slightly better than \cref{thm:main1} in case we add the assumption that the polynomials $\vec{f}$ are homogeneous:

\begin{thm}\label{thm:main2}
Suppose $\tilde{f}_i\in \z[\vec{x}]$ for $i=1,\ldots, r$ are homogeneous polynomials of degree $d$ so that
$K-r(r+1)(d-1)>0$,
$\vec{\tilde{f}}$ has a zero over $\z_p$ for all primes $p$ and a real zero. Assume that the corresponding projective variety $\widetilde{V}$ is non-singular. Then there exists an $\vec{x}\in \z^n\backslash\{0\}$, polynomially bounded by $C$ and $\widetilde{C}$, such that $\tilde{f}(\vec{x})=\vec{0}$, in fact
\begin{equation} \label{eq:mthom}
|\vec{x}|\ll \widetilde{C}^{6n^2r^n(d-1)^{n-2}d\cdot\frac{K+r(r+1)(d-1)}{K-r(r+1)(d-1)}}.
\end{equation}
\end{thm}
\begin{proof}
As in the proof of \cref{thm:main1} (with $C=\widetilde{C}$) we use that for $P$ satisfying
$$P\gg \left(\frac{\widetilde{C}^{K/(d-1)+r^2}}{\mathfrak{S}(\vec{0})J(\vec{0})}\right)^{1/\delta}$$
we have that $M(P,\vec{0})> 1$ (if the implied constant is large enough). Moreover, 
The quantitative version of the Nullstellensatz for $\vec{\tilde{f}}$ given in (\ref{eq:defctilde}) does still hold.
Hence, mutatis mutandis,
the proof of \cref{prop:s0} applies and we find that $\mathfrak{S}(\vec{0}) \geq \tilde{N}^{-3(n-r)}$. Together with \cref{cor:J0} and (\ref{eq:defc}) it follows that
\begin{align*}
\mathfrak{S}(\vec{0})J(\vec{0}) &\gg \tilde{\mathfrak{C}}^{-(2(n-r)-1)}\widetilde{C}^{-2r(n-r)}\tilde{N}^{-n+r-1} \\
&\gg \tilde{\mathfrak{C}}^{-(3(n-r)-2)}\widetilde{C}^{-2r(n-r)}.
\end{align*}
One finds that one can take
\begin{align*}
P&\gg\widetilde{C}^{6n(n-r)r^n(d-1)^{n-2}d\cdot\frac{K+r(r+1)(d-1)}{K-r(r+1)(d-1)}}.\qedhere\end{align*}
\end{proof}

As already indicated in the introduction, we provide a quantitative strong approximation theorem for systems $\vec{f}$ satisfying the same conditions as in \cref{thm:main1}. Call $\vec{x}\in \r^n$ totally positive if $x_i>0$ for all $i$. 

\begin{thm}\label{thm:cormain}
Let $\vec{m}, \vec{M}\in \z^n$. Suppose $f_i\in \z[\vec{x}]$ for $i=1,\ldots, r$ are polynomials of degree $d$ so that
$K-r(r+1)(d-1)>0$ and the corresponding varieties $V$ and $\widetilde{V}$ are non-singular affine respectively projective varieties. Suppose that a zero $\vec{y}\in \z_p$ of $\vec{f}$ satisfying
$y_i\equiv m_i \mod M_i$ exists for every prime $p$ and suppose $\vec{\tilde{f}}$ has a totally positive real zero. Then, there exists an $\vec{x}\in \z_{>0}^n$, polynomially bounded by $C$ and $\widetilde{C}$, such that 
$$f(\vec{x})=\vec{0} \quad \text{and} \quad x_i \equiv m_i \mod M_i$$
and
$$|\vec{x}|\ll (|\vec{M}|^{5d}C^3\widetilde{C}^2)^{2n^2r^{n+1}(d-1)^{n-1}d\cdot\frac{K+r(r+1)(d-1)}{K-r(r+1)(d-1)}},$$
where the implied constant does not depend on $C, \widetilde{C}, \vec{m}$ or $\vec{M}$. 
\end{thm}

\begin{proof}
Let
$$\vec{g}(\vec{y})=\vec{f}(\vec{My}+\vec{m}) \quad \text{and} \quad \vec{\tilde{g}}(\vec{y})=\widetilde{\vec{f}(\vec{My}+\vec{m})}=\tilde{f}(\vec{My}),$$
 where $(\vec{My})_i=M_iy_i$. Observe that over $\overline{\q}$ we have that $\vec{f}$ is non-singular if and only if $\vec{g}$ is non-singular and similarly for $\vec{\widetilde{f}}$. Moreover, the condition on the existence of zeros of $\vec{f}$ ensures that $\vec{g}$ has zeros over $\z_p$ for all primes $p$ and that $\vec{\tilde{g}}$ has a totally positive zero over $\mathbb{R}$. After scaling, this zero lies in $\mathcal{B}\subset (0,1]^n$. Now, apply \cref{thm:main1}. The statement follows by noting that the maximal coefficient of $\vec{g}$ and $\vec{\tilde{g}}$ are $\ll |\vec{M}|^dC$ and $\ll|\vec{M}|^d\widetilde{C}$ respectively as we can assume without loss of generality that $|m_i|\leq |M_i|$ for all $i=1,\ldots, n$.
\end{proof}

\end{document}